\documentclass[11pt]{article}

\usepackage{amsmath,amssymb,amsthm,amscd,epsfig}
\newcommand{\vp}{\varepsilon}

\newcommand{\ovl}{\overline}

\newcommand{\ul}{\underline}

\theoremstyle{plain}

\newtheorem{thm}{Theorem}

\newtheorem{lem}{Lemma}
\newtheorem{cor}{Corollary}

\theoremstyle{definition}

\theoremstyle{remark}

\begin{document}

\title{Pointwise dimension for a class of measures on limit sets}

\author{Eugen Mihailescu}

\date{}
\maketitle

\begin{abstract}
We study the pointwise dimension for a new class of projection measures on arbitrary fractal limit sets without separation conditions. We prove that the pointwise dimension exists a.e. for this class of measures associated to equilibrium states, and it is given by a formula in terms of Lyapunov exponents and a certain type of entropy. Thus these measures are exact dimensional. Self-conformal measures belong to the above class of measures, and this allows us to obtain a new geometric formula for their pointwise dimension. Thus for self-conformal measures we obtain also a geometric formula for their projection entropy.
\end{abstract}

\textbf{Mathematics Subject Classification 2010:} 37D35, 37D20, 37C45, 37A35, 28A80.

\textbf{Keywords:} Pointwise dimension; exact dimensional measures; densities; projection entropy; hyperbolic endomorphisms; self-conformal measures; Lyapunov exponents.

\section{Introduction.}

We study  the pointwise dimension and the exact dimensionality for a new class of  probability measures on fractal limit sets of arbitrary conformal iterated function systems. We obtain a general formula for the pointwise dimension of these new measures associated to equilibrium states.  Moreover we investigate the intricate process of interlacing of the generic iterates and the gaps consisting of  non-generic iterates in a typical trajectory, and how they influence the local densities of these measures and their pointwise  dimensions. As a particular case we obtain a geometric formula for the pointwise dimension of self-conformal measures for systems without separation conditions. For self-conformal measures our geometric formula relates their projection entropy (defined in \cite{FH}), with the average rate of growth for the generic number of overlappings in the limit set. Therefore our result gives a \textit{geometric interpretation} for the pointwise dimension and for the projection entropy in the case of self-conformal measures, and allows to estimate these quantities more directly.

The general setting is the following: \
let $\mathcal S = \{\phi_i, i \in I\}$ be an arbitrary finite iterated function system of smooth conformal injective contractions of a compact set with nonempty interior $V \subset \mathbb R^D$. We do not assume any kind of separation condition for $\mathcal S$. The limit set of $\mathcal S$ is then given by: $$\Lambda = \mathop{\cup}_{\omega \in \Sigma_I^+}\mathop{\cap}\limits_{n \ge 0} \phi_{\omega_1\ldots \omega_n}(V),$$
where $\Sigma_I^+$ is the 1-sided symbolic space on $|I|$ symbols, and $\omega = (\omega_1, \ldots, \omega_n, \ldots) \in \Sigma_I^+$ arbitrary.  \ 
\newline
 Denote by $[\omega_1\ldots\omega_n]$ the cylinder on the first $n$ elements of $\omega$, and by $\phi_{i_1\ldots i_p}:= \phi_{i_1}\circ \ldots \circ \phi_{i_p}$. The shift $\sigma: \Sigma_I^+ \to \Sigma_I^+$ is given by $\sigma(\omega) = (\omega_2, \omega_3, \ldots),  \omega \in \Sigma_I^+$. Also let the canonical projection  $$\pi: \Sigma_I^+ \to \Lambda, \ \pi(\omega) := \phi_{\omega_1\omega_2\ldots}(V)$$ 
Now consider the endomorphism which associates a dynamical system to $\mathcal S$, 
$$
\Phi: \Sigma_I^+ \times \Lambda \to \Sigma_I^+ \times \Lambda, \  \ \Phi(\omega, x) = (\sigma \omega, \phi_{\omega_1}(x)), \ (\omega, x) \in \Sigma_I^+ \times \Lambda$$
If $\pi_1$ is the projection on the first coordinate of $\Sigma_I^+\times \Lambda$, then the following diagram is commutative,
\begin{equation}\label{dia}
\begin{array}{clclcr}
\Sigma_I^+\times \Lambda & \ \ &\mathop{\longrightarrow}\limits^{ \Phi} & \ \  &\Sigma_I^+\times \Lambda \\
\pi_1\downarrow &   & \ \ \ &  & \downarrow \pi_1\\
\Sigma_I^+& \ \ & \mathop{\longrightarrow}\limits^{\sigma} & \ \ & \Sigma_I^+
\end{array}
\end{equation}
The endomorphism $\Phi$ has  a type of hyperbolic structure, since it is expanding in the first coordinate and contracting in the second coordinate (due to the uniform contractions in $\mathcal S$).
\newline
Consider a H\"older continuous potential $\psi: \Sigma_I^+ \to \mathbb R$, and let $\mu^+$ be the equilibrium measure of $\psi$ on $\Sigma_I^+$. Define $\hat\psi:= \psi \circ \pi_1: \Sigma_I^+ \times \Lambda \to \mathbb R$, which is also H\"older continuous. Thus since $\Phi$ is hyperbolic,  there is a unique equilibrium state $\hat \mu:= \mu_{\hat\psi}$ of $\hat\psi$ on $\Sigma_I^+ \times \Lambda$ (as in \cite{Bo}, \cite{KH}). Hence $\pi_{1*}\hat \mu = \mu^+$. \
\newline
Moreover we have the projection on the second coordinate, 
$$\pi_2: \Sigma_I^+\times \Lambda \to \Lambda, \ \pi_2(\omega, x) = x$$ 

The \textbf{main focus} of this paper is the  measure $\pi_{2*}\hat\mu$, which in general is \textit{different} from the canonical projection measure $\pi_{1*}\hat\mu$. 
Let us denote these two measures on $\Lambda$ by, 
\begin{equation}\label{mu1}
\nu_1:= (\pi\circ \pi_1)_*\hat\mu = \pi_*\mu^+, \ \ \text{and} \ \ 
\nu_2:= \pi_{2*}\hat\mu
\end{equation}
Since $\Phi^n(\omega, x) = (\sigma^n\omega, \phi_{\omega_n\ldots \omega_1}(x))$ reverses the order of $\omega_1, \ldots, \omega_n$ in its second coordinate, and since $\hat\mu$ is $\Phi^n$-invariant, we call $\nu_2=\pi_{2, *}\hat\mu$ an \textit{order-reversing projection measure}. This is in contrast with the construction of $\nu_1$.

Some important notions in Dimension Theory are those of lower/upper pointwise dimensions of a measure, and the notion of exact dimensional measures (see for eg \cite{Pe}). In general for a probability Borel measure $\mu$ on a metric space $X$,  the \textit{lower pointwise dimension} of $\mu$ at $x\in X$ is defined as: $$\underline\delta(\mu)(x):=\mathop{\liminf}\limits_{r \to 0} \frac{\log\mu(B(x, r))}{\log r},$$  and the \textit{upper pointwise dimension} of $\mu$ at $x\in X$ is: $$\overline{\delta}(\mu)(x):= \mathop{\limsup}\limits_{r\to 0} \frac{\log\mu(B(x, r))}{\log r}$$ If $\underline\delta(\mu)(x) = \overline\delta(\mu)(x)$ then we call the common value the \textit{pointwise dimension} of $\mu$ at $x$, denoted by $\delta(\mu)(x)$. If for $\mu$-a.e $x \in X$, the pointwise dimension $\delta(\mu)(x)$ exists and is constant, we say that $\mu$ is \textit{exact dimensional}. In this case there is a value $\alpha\in \mathbb R$ s.t for $\mu$-a.e $x\in X$, $$\delta(\mu)(x):=\underline\delta(\mu)(x) = \overline\delta(\mu)(x) = \alpha$$

The problem of pointwise dimensions for various types of invariant measures  has been studied in detail by many authors in different settings, such as in \cite{BPS}, \cite{FH}, \cite{HH}, \cite{Pe}, \cite{PW}, \cite{Sh}, \cite{Y}, \cite{FM}, \cite{M-ETDS11},  and this list is far from complete.

In \cite{FH} Feng and Hu introduced the notion of \textit{projection entropy} for an arbitrary $\sigma$-invariant probability measure $\mu$ on $\Sigma_I^+$, namely
\begin{equation}\label{projent} 
h_\pi(\sigma, \mu) := H_\mu(\mathcal P|\sigma^{-1}\pi^{-1}\gamma) - H_\mu(\mathcal P|\pi^{-1}\gamma),
\end{equation}
where $\mathcal P$ is the partition with 0-cylinders $\{[i], i \in I\}$ of $\Sigma_I^+$, and $\gamma$ is the $\sigma$-algebra of Borel sets in $\mathbb R^d$. If $\mu$ is ergodic, then it was shown in \cite{FH} that the measure $\pi_*\mu$ is exact dimensional on $\Lambda$, and that for $\mu$-a.e $\omega \in \Sigma_I^+$, its pointwise dimension is given by:  
\begin{equation}\label{FHformula}
\delta(\pi_*\mu)(\pi\omega) = \frac{h_\pi(\sigma, \mu)}{-\int_{\Sigma_I^+}\log |\phi_{\omega_1}(\pi\sigma\omega)| \ d\mu(\omega)}
\end{equation}
  In our case, this implies that  $\nu_1= \pi_*\mu^+$ is exact dimensional. 
  \
  Notice that in general, $\nu_1$ is not equal to $\nu_2$. So  the problem of the pointwise dimension of $\nu_2$ must be studied separately, and we do this in the sequel. 
  
Let us denote the \textit{stable Lyapunov exponent} of the endomorphism $\Phi$ with respect to $\hat\mu$ by, $$\chi_s(\hat\mu):= \int_{\Sigma_I^+\times\Lambda}\log|\phi_{\omega_1}(x)| \ d\hat\mu(\omega, x)$$

We will use the Jacobian in the sense of Parry \cite{Pa}; so let us consider the Jacobian $J_\Phi(\hat\mu)$ of a $\Phi$-invariant measure $\hat \mu$ on $\Sigma_I^+ \times \Lambda$ (see also \cite{BH}). It is clear that for $\hat\mu$-a.e $(\omega, x) \in \Sigma_I^+\times \Lambda$, we have $J_\Phi(\hat \mu) \ge 1$. From definition we have for $\hat\mu$-a.e $(\omega, x) \in \Sigma_I^+ \times \Lambda$, $$J_\Phi(\hat\mu)(\omega, x) = \mathop{\lim}\limits_{r\to 0} \frac{\hat\mu(\Phi(B((\omega, x), r)))}{\hat\mu(B((\omega, x), r))}$$
From the Chain Rule for Jacobians,  
$J_{\Phi^n}(\hat\mu)(\omega, x) = J_\Phi(\hat \mu)(\Phi^{n-1}(\omega, x))\cdot \ldots \cdot J_\Phi(\hat\mu)(\omega, x)$ for  $n \ge 1$,
and from the Birkhoff Ergodic Theorem applied to the integrable function $\log J_\Phi(\hat\mu)(\cdot, \cdot)$, we have that for $\hat\mu$-a.e $(\omega, x) \in \Sigma_I^+\times \Lambda$,
\begin{equation}\label{BET}
\frac{\log J_{\Phi^n}(\hat \mu)(\omega, x)}{n} \mathop{\longrightarrow}\limits_{n \to \infty} \int_{\Sigma_I^+\times \Lambda} \log J_{\Phi}(\hat\mu)(\eta, y) \ d\hat\mu(\eta, y)
\end{equation}
Ruelle introduced (\cite{Ru-fold}, \cite{Ru-survey}) the notion of \textit{folding entropy} $F_f(\nu)$ of a measure $\nu$ invariant with respect to an endomorphism $f: X \to X$ on a Lebesgue space $X$, as the conditional entropy $H_\nu(\epsilon|f^{-1}\epsilon)$, where $\epsilon$ is the point partition of $X$ and $f^{-1}\epsilon$ is the partition with the fibers of $f$. In fact from \cite{Pa}, \cite{Ru-fold} it follows that 
$F_f(\nu) = \int_X \log J_f(\nu) d\nu$.
Thus in our case for the $\Phi$-invariant measure $\hat \mu$ on $\Sigma_I^+\times \Lambda$, we have
\begin{equation}\label{FlogJ}
F_\Phi(\hat \mu)  = \int_{\Sigma_I^+\times \Lambda} \log J_\Phi(\hat \mu) \ d\hat \mu
\end{equation}

\

The main result of the current paper is Theorem \ref{thm1}, saying that the pointwise dimension of $\nu_2$ is related to the folding entropy of the lift measure $\hat\mu$. Recall that in general $\nu_2$ is not equal to $\nu_1:= \pi_{1*}\hat\mu$, \ thus usually $\delta(\nu_2)$ is not given by  (\ref{FHformula}).

\begin{thm}\label{thm1}
Let $\mathcal S$ be a finite conformal iterated function system, and $\psi$ be a H\"older continuous potential on $\Sigma_I^+$ with equilibrium measure $\mu^+$, and denote by $\hat \mu$ the equilibrium measure of $\psi \circ \pi_1$ on $\Sigma_I^+\times \Lambda$ with respect to $\Phi$. Denote  by $\nu_2:= \pi_{2*}\hat\mu$. Then for $\nu_2$-a.e $x\in \Lambda$,
$$\delta(\nu_2)(x) = \frac{F_\Phi(\hat\mu) - h_\sigma(\mu^+)}{\chi_s(\hat\mu)}$$
In particular,  the measure $\nu_2$ is exact dimensional on $\Lambda$.
\end{thm}

In our case the folding entropy turns out to be related to the overlap number of $\hat\mu$. 
The notion of \textit{overlap number} $o(\mathcal S, \mu_g)$ for an equilibrium measure $\mu_g$ of a H\"older continuous potential $g:\Sigma_I^+ \times \Lambda \to \mathbb R$ was introduced in \cite{MU-JSP2016}, and represents an average asymptotic rate of growth for the number of $\mu_g$-generic overlaps of order $n$ in $\Lambda$. Namely for any $\tau>0$ let the set of $\tau$-generic preimages with respect to $\mu_g$ having the same $n$-iterates as $(\omega, x)$, $$\Delta_n((\omega, x), \tau, \mu_g) := \{(\eta_1, \ldots, \eta_n) \in I^n, \exists y \in \Lambda, \phi_{\omega_n\ldots \omega_1}(x) = \phi_{\eta_n\ldots \eta_1}(y),  \ |\frac{S_ng(\eta, y)}{n} - \int_{\Sigma_I^+ \times \Lambda} g\ d\mu_\psi| < \tau\},$$
where $(\omega, x) \in \Sigma_I^+ \times \Lambda$ and $S_n g(\eta, y)$ is the consecutive sum of $g$ with respect to $\Phi$. Denote  by $$b_n((\omega, x), \tau, \mu_g):= Card \Delta_n((\omega, x), \tau, \mu_g)$$
Then, in \cite{MU-JSP2016} we showed that the following limit exists and defines the \textit{overlap number} of $\mu_g$,
$$o(\mathcal S, \mu_g) = \mathop{\lim}\limits_{\tau \to 0} \mathop{\lim}\limits_{n \to \infty} \frac 1n \int_{\Sigma_I^+ \times \Lambda} \log b_n((\omega, x), \tau, \mu_g) \ d\mu_g(\omega, x)$$ 
Moreover, there is a relation between the overlap number and the folding entropy of $\mu_g$, 
\begin{equation}\label{oS}
o(\mathcal S, \mu_g) = \exp(F_\Phi(\mu_g))
\end{equation} 
When $\mu_g = \mu_0$ is the measure of maximal entropy, we denote $o(\mathcal S, \mu_0)$ by $o(\mathcal S)$ and call it \textit{the topological overlap number} of $\mathcal S$. Notice  that all preimages are generic for the measure of maximal entropy, so $o(\mathcal S)$ represents an asymptotic rate of growth of the total number of overlaps between the $n$-iterates of type $\phi_{i_1\ldots i_n}(\Lambda)$, when $i_1, \ldots, i_n \in I$ and $n \to \infty$.
\newline
Combining (\ref{oS}) and Theorem \ref{thm1} we obtain a \textit{formula for the pointwise dimension} of $\nu_2$ in terms of its overlap number:

\begin{cor}\label{ovlnu2}
In the setting of Theorem \ref{thm1}, the pointwise dimension of  $\nu_2$ satisfies for $\nu_2$-a.e $x \in \Lambda$, $$\delta(\nu_2)(x) = \frac{\exp(o(\mathcal S, \hat \mu)) - h_\sigma(\mu^+)}{\chi_s(\hat\mu)}$$
\end{cor}

\

In particular, $o(\mathcal S, \hat \mu)$ and thus $\delta(\nu_2)$ can be easily computed above if for some $m \ge 1$, there exists a constant number of overlaps between the sets of type $\phi_{i_1\ldots i_m}(\Lambda)$, modulo $\hat\mu$.

\

An important case is  that of \textit{self-conformal measures}, i.e $\pi_1$-projections of Bernoulli measures on $\Lambda$.
To fix notation, given the system $\mathcal S$ and the probability vector $\textbf p = (p_1, \ldots, p_{|I|})$, let $\nu_{\textbf p}$ be the corresponding Bernoulli measure on $\Sigma_I^+$. Any Bernoulli measure $\nu_{\textbf p}$ on $\Sigma_I^+$ is the equilibrium measure of some H\"older continuous potential $\psi_{\textbf p}$. In this case the Bernoulli measure $\nu_{\textbf p}$ is the measure  $\mu^+$ from before.  Thus we have the above construction and results. 
 Then denote by $\hat\mu_{\textbf p}$ the lift of $\nu_{\textbf p}$ to $\Sigma_I^+\times \Lambda$ obtained as the equilibrium measure of $\psi_{\textbf p} \circ \pi_1$, and by $\nu_{1, \textbf p}, \nu_{2, \textbf p}$, the associated projected measures $\nu_1, \nu_2$.
\newline
For \textit{self-conformal measures}  (i.e projections $\nu_{1, \textbf p} = \pi_*\nu_{\textbf p}$ of Bernoulli measures $\nu_{\textbf p}$ on $\Sigma_I^+$), we showed in \cite{MU-JSP2016}  that the measures $\nu_{1, \textbf p}$ and $\nu_{2, \textbf p}$ are equal. Therefore, $\delta(\nu_{1, \textbf p}) = \delta(\nu_{2, \textbf p}).$ 
We obtain then below a \textit{formula for the projection entropy} $h_\pi(\sigma, \nu_{\textbf p})$ of  the measure $\nu_{1, \textbf{p}}$ in terms of its overlap number.

\begin{thm}\label{thm2}
In the setting of Theorem \ref{thm1}, let $\nu_{\textbf p}$ be a Bernoulli measure on $\Sigma_I^+$, and let $\nu_{1, \textbf p}, \nu_{2, \textbf p}$ be the associated projection measures on $\Lambda$.
Then $\nu_{1, \textbf p} = \nu_{2, \textbf p}$, and for $\nu_{1, \textbf p}$-a.e $x\in \Lambda$, $$\delta(\nu_{1, \textbf p})(x) = \delta(\nu_{2, \textbf p})(x) = \frac{\exp(o(\mathcal S, \hat \mu_{\textbf p})) - h_\sigma(\nu_{\textbf p})}{\chi_s(\hat\mu_{\textbf p})} $$ 
Moreover the projection entropy of $\nu_{\textbf p}$ is determined by, $$h_\pi(\sigma, \nu_{\textbf p}) = h_\sigma(\nu_{\textbf p}) - \exp(o(\mathcal S, \hat\mu_{\textbf p}))$$
In particular the projection entropy of  the measure $\mu_0$ of maximal entropy on $\Sigma_I^+$ is obtained as:
$$h_\pi(\sigma, \mu_0) = \log |I| - \exp(o(\mathcal S))$$
\end{thm}

\section{Main Results and Proofs.}

Recall the setting from Section 1, where $\psi$ is a H\"older continuous potential on $\Sigma_I^+$, $\mu^+$ is the equilibrium measure of $\psi$ on $\Sigma_I^+$, and $\hat\mu$ denotes the equilibrium measure $\mu_{\hat\psi}$ of $\hat\psi:=\psi\circ \pi_1$ on $\Sigma_I^+\times \Lambda$. \
We consider the measurable partition $\xi$ of $\Sigma_I^+\times \Lambda$ with the fibers of the projection $\pi_1: \Sigma_I^+\times \Lambda \to \Lambda$, and the associated conditional measures $\mu_\omega$ of $\hat \mu= \mu_{\hat\psi}$ defined for $\mu^+$-a.e $\omega \in \Sigma_I^+$ (see \cite{Ro}); from above, $\mu^+ = \pi_{1*}\hat\mu$. For $\mu^+$-a.e $\omega \in \Sigma_I^+$, the conditional measure $\mu_\omega$ is defined on $\pi_1^{-1}\omega = \{\omega\}\times\Lambda$.
It is clear that the factor space $\Sigma_I^+\times \Lambda/\xi$ is equal to $\Sigma_I^+$, and the corresponding factor measure of $\hat\mu$ satisfies, $$\hat\mu_\xi(A) = \hat\mu(A\times \Lambda) = \mu^+(A),$$
for any measurable set $A \subset \Sigma_I^+$. Thus $\hat\mu_\xi = \mu^+$.
From the properties of the conditional measures, we obtain that  for any borelian set $E$ in $\Sigma_I^+\times \Lambda$, 
\begin{equation}\label{desint}
\hat\mu(E) = \int_{\pi_1E}(\int_{\{\omega\}\times\Lambda}\chi_E d\mu_\omega) \ d\mu^+(\omega) = \int_{\pi_1E} \mu_\omega(E \cap \{\omega\}\times\Lambda) \ d\mu^+(\omega)
\end{equation}
For a Borel set $A$ in $\Lambda$, we have for $\mu^+$-a.e $\omega \in \Sigma_I^+$, 
\begin{equation}\label{muomega}
\mu_\omega(A) = \mathop{\lim}\limits_{n\to \infty}\frac{\hat\mu([\omega_1\ldots \omega_n]\times A)}{\mu^+([\omega_1\ldots \omega_n])}
\end{equation}
\newline
\textit{Notation.} 
Two quantities $Q_1, Q_2$ are called \textit{comparable} $Q_1 \approx Q_2$, if there is a constant $C>0$ with $$\frac 1C Q_1 \le Q_2 \le C Q_1$$ In general the comparability constant $C$ is independent of the parameters appearing in $Q_1, Q_2$. 

$\hfill\square$

The above conditional measures $\mu_\omega$ are defined on $\{\omega\}\times\Lambda$, so actually they can be considered as probability measures on $\Lambda$. In the next Lemma, we compare $\mu_\omega(A)$ with $\mu_\eta(A)$, and show that $\hat\mu$ has an ``almost'' product structure with respect to $\mu^+$ and $\mu_\omega$.

\begin{lem}\label{prodstr}
There exists a constant $C>0$ so that  for $\mu^+$-a.e $\omega, \eta \in \Sigma_I^+$ and any Borel set $A \subset \Lambda$, 
$$\frac 1C \mu_\eta(A) \le \mu_\omega(A) \le C\mu_\eta(A)$$
Moreover for any Borel sets $A_1 \subset \Sigma_I^+, A_2 \subset \Lambda$ and for $\mu^+$-a.e $\omega\in \Sigma_I^+$, we have:
$$\frac 1C \mu^+(A_1)\cdot\mu_\omega(A_2) \le \hat\mu(A_1 \times A_2) \le C \mu^+(A_1) \cdot \mu_\omega(A_2)$$
In particular there is a constant $C>0$ such that, for $\mu^+$-a.e $\omega \in \Sigma_I^+$ and any Borel set $A\subset \Lambda$,  
$$\frac 1C \mu_\omega(A) \le \nu_2(A) \le C \mu_\omega(A)$$  
\end{lem} 

\begin{proof}
First recall formula (\ref{muomega}) for the conditional measure $\mu_\omega$. From the $\Phi$-invariance of $\hat\mu$, 
\begin{equation}\label{inv}
\hat\mu([\omega_1\ldots \omega_n]\times A) = \mathop{\sum}\limits_{i \in I} \hat\mu([i\omega_1\ldots \omega_n]\times \phi_i^{-1}A)
\end{equation}
Now we can cover the set $A$ with small disjoint balls (modulo $\hat\mu$), so it is enough to consider such a small ball $B = A \subset \Lambda$. The general case will follow then from this.
\newline
Recall also that for any $i_1, \ldots, i_n \in I, n \ge 1$,  $\phi_{i_1\ldots i_n}$ denotes the composition $\phi_{i_1}\circ \ldots \circ \phi_{i_n}$. We have a Bounded Distortion Property, since we work with smooth conformal contractions $\phi_i$, i.e there exists a constant $C>0$ such that for any $x, y, n, i_1, \ldots, i_n$, we have $|\phi'_{i_1\ldots i_n}(x)| \le C |\phi_{i_1\ldots i_n}'(y)|$. Since the contractions $\phi_i$ are conformal, let  $i_1, \ldots, i_p \in I$ such that $\phi_{i_p}^{-1}\ldots \phi_{i_1}^{-1}B = \phi_{i_1\ldots i_p}^{-1}B$ is a ball $B(x_0, r_0)$ of a fixed radius $r_0$.
In this way we inflate $B$ along any backward trajectory $\underline i = (i_1, i_2, \ldots) \in \Sigma_I^+$ up to some maximal order $p(\underline i) \ge 1$, so that $\phi_{i_1\ldots i_{p(\underline i)}}^{-1} B $ contains a ball of radius $C_1 r_0$ and it is contained in a ball of radius $ r_0$, for some fixed constant $C_1$ independent of $B, \underline i$. Then by using successively the $\Phi$-invariance of $\hat\mu$, relation (\ref{inv}) becomes:
\begin{equation}\label{inve}
\hat\mu([\omega_1\ldots \omega_n]\times B) = \mathop{\sum}\limits_{\underline i \in I} \hat\mu\big([i_{p(\underline i)}\ldots i_1 \omega_1\ldots \omega_n]\times \phi_{i_1\ldots i_{p(\underline i)}}^{-1}B\big)
\end{equation}
Without loss of generality one can assume that  $\phi_{i_1\ldots i_{p(\underline i)}}^{-1} B $ is a ball of radius $r_0$. Notice that the set $[i_{p(\underline i)}\ldots i_1\omega_1\ldots \omega_n]\times \phi_{i_1\ldots i_{p(\underline i)}}^{-1}B$ is the Bowen ball $[i_{p(\underline i)}\ldots i_1\omega_1\ldots \omega_n] \times B(x_0, r_0)$ for $\Phi$. Since $\hat\mu$ is the equilibrium state of $\psi \circ \pi_1$, and since $P_\Phi(\psi \circ \pi_1) = P_\sigma(\psi):= P(\psi)$, we have:
\begin{equation}\label{Bx0}
\begin{aligned}
\hat\mu([i_{p(\underline i)}\ldots i_1\omega_1\ldots \omega_n] \times \phi_{i_1\ldots i_{p(\underline i})}^{-1}B) &\approx \exp(S_{n+p(\underline i)}\psi(i_{p(\underline i)}\ldots i_1\omega_1\ldots \omega_n)-(n+p(\underline i))P(\psi))\approx\\
&\approx \mu^+([\omega_1\ldots \omega_n]) \cdot \hat \mu([i_{p(\underline i)}\ldots i_1]\times \phi_{i_1\ldots i_{p(\underline i)}}^{-1}B),
\end{aligned}
\end{equation}
where the comparability constant above does not depend on $B, i_1, \ldots, i_{p(\underline i)}, n$.
If we choose another finite sequence $(\eta_1, \ldots, \eta_n) \in I^n$, then we can take again for any $\underline i \in \Sigma_I^+$, the same indices $i_1, \ldots, i_{p(\underline i)}$ such that $\phi_{i_1\ldots i_{p(\underline i)}}^{-1}B$ is a ball of radius $r_0$, thus,
\begin{equation}\label{Bx1}
\begin{aligned}
\hat\mu([i_{p(\underline i)}\ldots i_1\eta_1\ldots \eta_n]\times \phi_{i_1\ldots i_{p(\underline i)}}^{-1} B) &\approx \exp(S_{n+p(\underline i)}\psi(i_{p(\underline i)}\ldots i_1\eta_1\ldots \eta_n) - (n+p(\underline i)) P(\psi)) \\
&\approx \mu^+([\eta_1\ldots \eta_n]) \cdot \hat\mu([i_{p(\underline i)}\ldots i_1]\times \phi_{i_1\ldots i_{p(\underline i)}}^{-1}B),
\end{aligned}
\end{equation}
where the comparability constant does not depend on $B, i_1, \ldots, i_{p(\underline i)}, n$.
But the cover of $A$ with small balls  of type $B$ and the above process of inflating these balls along prehistories $\underline i$ to balls of radius $r_0$, can be done along any trajectories $\omega, \eta$. 
Thus by  (\ref{inve}) and using the uniform estimates (\ref{Bx0}), (\ref{Bx1}) and (\ref{muomega}), we obtain that there exists a constant $C>0$ such that for $\mu^+$-a.e $\omega, \eta \in \Sigma_I^+$, 
\begin{equation}\label{etao}
\frac 1C\mu_\eta(A) \le \mu_\omega(A) \le C\mu_\eta(A)
\end{equation}
Thus from (\ref{etao}) and the desintegration formula (\ref{desint}) for $\hat\mu$, it follows
 that there exists a constant (denoted also by $C$) such that for all Borel sets $A_1\subset \Sigma_I^+, A_2\subset \Lambda$, 
 $$
\frac 1C \mu^+(A_1)\cdot\mu_\omega(A_2)\le \hat\mu(A_1\times A_2) \le C \mu^+(A_1) \cdot \mu_\omega(A_2)
$$
For the final statement, recall that $\nu_2 = \pi_{2*}\hat\mu$, so $\nu_2(A) = \hat\mu(\Sigma_I^+\times A)$. Then we use the last displayed formula to obtain a constant $C>0$, such that for $\mu^+$-a.e $\omega \in \Sigma_I^+$ and any Borel set $A\subset \Lambda$,  we have $$\frac 1C \mu_\omega(A) \le \nu_2(A) \le C \mu_\omega(A)$$
\end{proof}
Now we prove Theorem \ref{thm1}, i.e. the formula for the pointwise dimension of $\nu_2$.

\

\textit{Proof of Theorem \ref{thm1}.}

First we prove the upper estimate for the pointwise dimension of $\nu_2$. For any $n \ge 1, (\omega, x) \in \Sigma_I^+\times \Lambda$, \ $\Phi^n(\omega, x) = (\sigma^n\omega, \phi_{\omega_n\ldots \omega_1}(x))$. From Birkhoff Ergodic Theorem applied to the $\Phi$-invariant measure $\hat\mu$, it follows that for $\hat\mu$-a.e $(\omega, x) \in \Sigma_I^+\times \Lambda$, 
$$\frac 1n \log|\phi_{\omega_n\ldots \omega_1}'|(x) \mathop{\longrightarrow}\limits_n \mathop{\sum}\limits_{i \in I} \int_{[i]\times \Lambda} \log|\phi_i'(x)| \ d\hat\mu(\omega, x) = \chi_s(\hat\mu)$$
On the other hand, from the Chain Rule for Jacobians, Birkhoff Ergodic Theorem and the formula for folding entropy $F_\Phi(\hat\mu)$, it follows that for $\hat\mu$-a.e $(\omega, x) \in \Sigma_I^+\times \Lambda$, 
$$\frac 1n \log J_{\Phi^n}(\hat\mu)(\omega, x) \mathop{\longrightarrow}\limits_n F_\Phi(\hat\mu)$$
 Thus for a set of $(\omega, x) \in\Sigma_I^+\times \Lambda$ of  full $\hat\mu$-measure,  
$$\frac 1n \log|\phi_{\omega_n\ldots \omega_1}'|(x) \mathop{\longrightarrow}\limits_n \mathop{\sum}\limits_{i \in I} \int_{[i]\times \Lambda} \log|\phi_i'(x)| \ d\hat\mu(\omega, x) = \chi_s(\hat\mu), \ \text{and} \  \ \frac 1n \log J_{\Phi^n}(\hat\mu)(\omega, x) \mathop{\longrightarrow}\limits_n F_\Phi(\hat\mu)$$
We now want to prove that the Jacobian $J_{\Phi^n}(\hat\mu)(\omega, x)$ depends basically  only on $\omega_1, \ldots, \omega_n$, i.e there exists a constant $C>0$ such that for every  $n \ge 1$, and $\hat\mu$-a.e  $(\eta, x)  \in [\omega_1\ldots \omega_n] \times \Lambda$, 
\begin{equation}\label{compjac}
\frac 1C J_{\Phi^n}(\hat\mu)(\eta, x) \le J_{\Phi^n}(\hat\mu)(\omega, x) \le C J_{\Phi^n}(\hat\mu)(\eta, x)
\end{equation}
In order to prove this, notice that if $r>0$, and $p>1$ is such that $diam[\omega_1\ldots \omega_{n+p}] = r$, then 
\begin{equation}\label{jacphin}
J_{\Phi^n}(\hat \mu)(\omega, x) = \mathop{\lim}\limits_{r\to 0, p \to \infty} \frac{\hat\mu(\Phi^n([\omega_1\ldots\omega_{n+p}]\times B(x, r))}{\hat\mu([\omega_1\ldots \omega_{n+p}]\times B(x, r))}
\end{equation}
 In our case, $\Phi^n([\omega_1\ldots \omega_{n+p}]\times B(x, r)) = [\omega_{n+1}\ldots \omega_{n+p}]\times \phi_{\omega_n\ldots\omega_1}B(x, r)$. If $\eta \in [\omega_1\ldots \omega_n]$,  
 $$\Phi^n([\eta_1\ldots \eta_{n+p}] \times B(x, r)) = [\eta_{n+1}\ldots \eta_{n+p}]\times \phi_{\omega_n \ldots \omega_1} B(x, r)$$
But from Lemma \ref{prodstr} there exists a constant $C>0$ such that for $\mu^+$-a.e $\omega \in \Sigma_I^+$, and any $n, p\ge 1$, 
\begin{equation}\label{comphat}
\begin{aligned}
\frac 1C \mu^+([\omega_{n+1}\ldots \omega_{n+p}])&\mu_\omega(\phi_{\omega_n\ldots\omega_1}B(x, r)) \le \hat \mu([\omega_{n+1}\ldots \omega_{n+p}]\times \phi_{\omega_n\ldots \omega_1}B(x, r)) \le \\
&\le C \mu^+([\omega_{n+1}\ldots \omega_{n+p}])\mu_\omega(\phi_{\omega_n\ldots \omega_1} B(x, r)),
\end{aligned}
\end{equation}
and similarly for $\hat\mu([\eta_{n+1}\ldots \eta_{n+p}]\times \phi_{\omega_n\ldots \omega_1} B(x, r))$.
Hence in view of (\ref{jacphin}) and (\ref{comphat}), we have only to compare the following quantities,
$$\frac{\mu^+([\omega_{n+1}\ldots\omega_{n+p}])\cdot \mu_\omega(\phi_{\omega_n\ldots \omega_1} B(x, r))}{\mu^+([\omega_1\ldots \omega_{n+p}])\cdot \mu_\omega(B(x, r))} \ \ \text{and} \ \ \frac{\mu^+([\eta_{n+1}\ldots \eta_{n+p}])\cdot\mu_\omega(\phi_{\omega_n\ldots\omega_1} B(x, r))}{\mu^+([\eta_1\ldots \eta_{n+p}]) \cdot \mu_\omega(B(x, r))}$$
However recall that $\eta \in [\omega_1\ldots \omega_n]$, thus there exists a constant $K>0$ such that 
\begin{equation}\label{holderS}
|S_n\psi(\eta_1\ldots \eta_n\ldots) - S_n\psi(\omega_1\ldots \omega_n\ldots)| \le K,
\end{equation}
since $\psi$ is H\"older continuous and $\sigma$ is expanding on $\Sigma_I^+$.
The same argument also implies that  \ $S_{n+p}\psi(\omega_1\ldots \omega_{n+p}\ldots)$ is determined in fact only by the first $n+p$ coordinates (modulo an additive constant).
Since $\mu^+$ is the equilibrium measure of $\psi$ on $\Sigma_I^+$, then 
\begin{equation}\label{mu+o}
\begin{aligned}
&\frac{\mu^+([\omega_{n+1}\ldots \omega_{n+p}])}{\mu^+([\omega_1\ldots \omega_{n+p}])} \approx \frac{\exp(S_p\psi(\omega_{n+1}\ldots \omega_{n+p}\ldots) - pP(\psi))}{\exp(S_{n+p}\psi(\omega_1\ldots \omega_{n+p}) - (n+p)P(\psi))} \ \text{and},\\
&\frac{\mu^+([\eta_{n+1}\ldots \eta_{n+p}])}{\mu^+([\eta_1\ldots \eta_{n+p}])} \approx \frac{\exp(S_p\psi(\eta_{n+1}\ldots \eta_{n+p}\ldots) - pP(\psi))}{\exp(S_{n+p}\psi(\eta_1\ldots \eta_{n+p}) - (n+p)P(\psi))},
\end{aligned}
\end{equation}
 where the comparability constant does not depend on $n, p, \omega, \eta$. But we have: \ $S_{n+p}\psi(\omega_1\ldots \omega_{n+p}\ldots) = S_n\psi(\omega_1\ldots \omega_{n+p}\ldots) + S_p\psi(\omega_{n+1}\ldots \omega_{n+p}\ldots)$.
 And similary for $S_{n+p}\psi(\eta_1\ldots \eta_{n+p}\ldots)$.
Therefore, using (\ref{jacphin}), (\ref{comphat}), (\ref{holderS}) and (\ref{mu+o}), we obtain the Jacobians inequalites in (\ref{compjac}).

Let us take now, for any $n >1$ and $\vp>0$,  the Borel set in $\Sigma_I^+\times\Lambda$: 
$$A(n, \vp):= \{(\omega, x), \big|\frac{\log|\phi_{\omega_n\ldots \omega_1}'(x)|}{n}-\chi_s(\hat \mu)\big| < \vp, \ |\frac{\log J_{\Phi^n}(\hat\mu)(\omega, x)}{n} - F_\Phi(\hat\mu)|<\vp,   \text{and} \ |\frac{S_n\psi(\omega)}{n}-\int\psi d\mu^+|<\vp\}$$
Then from Birkhoff Ergodic Theorem, for any $\vp>0$, $\hat\mu(A(n, \vp)) \mathop{\longrightarrow}\limits_{n\to \infty} 1$. From (\ref{compjac}), if $(\omega, x) \in A(n, \vp)$ and $\eta \in [\omega_1\ldots \omega_n]$, then $(\eta, x) \in A(n, 2\vp)$, so for any $\delta>0$, 
$$
\mu^+(\{\omega \in \Sigma_I^+, \nu_2(\pi_2(A(n, \vp)\cap [\omega_1\ldots \omega_n]\times \Lambda))>1-\delta\}) \mathop{\longrightarrow}\limits_{n} 1$$
We have from (\ref{compjac}) that $[\omega_1\ldots\omega_n] \times \pi_2 A(n, \vp)\subset A(n, 2\vp)$. Thus from Lemma \ref{prodstr}, for $n > n(\delta)$, 
 \begin{equation}\label{hatmuA}
 \hat\mu(A(n, 2\vp)\cap [\omega_1\ldots \omega_n]\times \Lambda) > C(1-\delta)\cdot\hat\mu([\omega_1\ldots \omega_n]\times \Lambda)
 \end{equation}

Let now $r_n:= 2\vp|\phi_{\omega_n\ldots \omega_1}'(x)|$, \ for $(\omega, x) \in A(n, \vp) \cap [\omega_1\ldots \omega_n]\times \Lambda$. With $y = \phi_{\omega_n\ldots \omega_1}(x)$ we have $\nu_2(B(y, r_n)) \ge \nu_2(\phi_{\omega_n\ldots \omega_1}(\pi_2 (A(n, 2\vp) \cap [\omega_1\ldots \omega_n]\times \Lambda))$, and then since $\nu_2 = \pi_{2*}\hat \mu$, we obtain
$$\nu_2(B(y, r_n)) \ge \hat\mu(\Sigma_I^+\times \pi_2(\Phi^n(A(n, 2\vp)\cap [\omega_1\ldots \omega_n]\times \Lambda))) \ge \hat\mu(\Phi^n(A(n, \vp)\cap [\omega_1\ldots \omega_n]\times \Lambda))$$
But  $\Phi^n$ is injective on the cylinder $[\omega_1\ldots \omega_n]\times \Lambda$ since $\phi_j$ are injective. Thus we can apply the Jacobian formula for the measure of the $\Phi^n$-iterate in the last term of last inequality above,
\begin{equation}\label{injf}
\begin{aligned}
&\nu_2(B(y, r_n)) \ge \hat\mu(\Phi^n(A(n, \vp)\cap [\omega_1\ldots \omega_n]\times \Lambda)) = \int_{A(n, \vp)\cap [\omega_1\ldots \omega_n]\times \Lambda} J_{\Phi^n}(\hat\mu) \ d\hat\mu \\
&\ge \exp\big(n(F_\Phi(\hat\mu)- \vp)\cdot \hat\mu(A(n, \vp) \cap [\omega_1\ldots \omega_n]\times \Lambda\big)
\end{aligned}
\end{equation} 
Recall that $\mu^+([\omega_1\ldots \omega_n]) = \hat \mu([\omega_1\ldots \omega_n]\times \Lambda)$. Then, from (\ref{injf}) and (\ref{hatmuA}) we obtain:
\begin{equation}\label{lastpager}
\nu_2(B(y, r_n))\ge \exp(n(F_\Phi(\hat\mu)-\vp))(1-\delta)C \mu^+([\omega_1\ldots \omega_n]) \ge C(1-\delta) e^{n(F_\Phi(\hat\mu)-\vp)}\exp(S_n\psi(\omega) - nP(\psi)),
\end{equation}
where $C$ is independent of $n, \omega, x, y$, and $P(\psi):=P_\sigma(\psi)$. Since $\mu^+$ is the equilibrium state of $\psi$,  $$P(\psi) = h(\mu^+) + \int \psi \ d\mu^+$$
But from the definition of $A(n, \vp)$, for any $(\omega, x) \in A(n, \vp)$, we have $$e^{n(\chi_s(\hat\mu) +\vp)} \ge  r_n=2\vp|\phi_{\omega_n\ldots\omega_1}(x)| \ge e^{n(\chi_s(\hat\mu) - \vp)}, \ \ \text{hence},$$ 
\begin{equation}\label{nchi}
n (\chi_s(\hat\mu)+\vp) \ge \log r_n \ge n(\chi_s(\hat\mu) - \vp)
\end{equation}
From (\ref{lastpager}), (\ref{nchi}) and the above formula for pressure, we obtain that for $\nu_2$-a.e $y \in \Lambda$,
\begin{equation}\label{upperestd}
\overline{\delta}(\nu_2)(y) = \mathop{\overline{\lim}}\limits_{r\to 0}\frac{\log \nu_2(B(y, r))}{\log r} \le \frac{F_\Phi(\hat\mu)-h(\mu^+)}{\chi_s(\hat\mu)} = \frac{h(\mu^+)-F_\Phi(\hat\mu)}{|\chi_s(\hat\mu)|},
\end{equation}
which proves the upper estimate for the (upper) pointwise dimension of $\nu_2$.

\

Now we prove the more difficult lower estimate for the pointwise dimension of $\nu_2$.
Define for any $m \ge 1$ and $\vp>0$, the following Borel set in $\Sigma_I^+\times \Lambda$, 
$$
\begin{aligned}
\tilde A(m, \vp) := \big\{(\omega, x) \in \Sigma_I^+\times \Lambda, &\ |\frac 1n\log|\phi_{\omega_n\ldots \omega_1}'(x)|-\chi_s(\hat\mu)|<\vp, \ \text{and} \ |\frac 1n \log J_{\Phi^n}(\hat\mu)(\omega, x)- F_{\Phi^n}(\hat\mu)| < \vp, \\
& \text{and} \ |\frac 1n S_n\psi(\omega) - \int \psi d\mu^+| < \vp, \ \forall n \ge m\big\}
\end{aligned}
$$ 
We know from Birkhoff Ergodic Theorem that $\hat\mu(\tilde A(m, \vp)) \mathop{\to}\limits_m 1$, for any $\vp>0$. So we obtain $\nu_2(\pi_2\tilde A(m, \vp)) \mathop{\to}\limits_m 1$. But $\Phi^n([i_1\ldots i_n]\times \Lambda) = \Sigma_I^+\times \phi_{i_n\ldots i_1}\Lambda$, and from the $\Phi$-invariance of $\hat\mu$, we have $\hat\mu(\Phi^n([i_1\ldots i_n]\times \Lambda) \ge \hat\mu([i_1\ldots i_n] \times \Lambda)$. Moreover $\hat\mu([i_1\ldots i_n] \times \Lambda) >0$,  since $\hat\mu$ is the equilibrium measure of $\psi\circ \pi_1$ and $[i_1\ldots i_n]\times \Lambda$ is an open set. In conclusion,  
\begin{equation}\label{pozm}
\nu_2(\phi_{i_n\ldots i_1}\Lambda) = \hat\mu(\Sigma_I^+\times \phi_{i_n\ldots i_1}\Lambda) = \hat\mu(\Phi^n([i_1\ldots i_n]\times \Lambda) \ge \hat\mu([i_1\ldots i_n] \times \Lambda) > 0
\end{equation}

We summarize now the general strategy of the proof, which will be detailed in the sequel. Since $\Phi^m([i_1\ldots i_m]\times \Lambda) = \Sigma_I^+\times \phi_{i_m\dots i_1}\Lambda$, and $\nu_2 = \pi_{2*}\hat\mu$, we have: $$\nu_2(\phi_{i_m\ldots i_1}\Lambda) = \hat\mu(\Sigma_I^+\times \phi_{i_m\ldots i_1}\Lambda) = \hat\mu(\Phi^m([i_1\ldots i_m]\times \Lambda)$$ 
Notice that a small ball $B(x, r)$ can intersect many sets of type $\phi_{i_m\ldots i_1}\Lambda$, for various $m$-tuples $(i_1, \ldots, i_m) \in I^m$, and these image sets may also intersect one another. Thus when estimating $\nu_2(B(x, r))$, all of these sets must be considered; it is not enough in principle to consider only one intersection $B(x, r) \cap \phi_{i_m\ldots i_1}\Lambda$.  \ However, we know from (\ref{pozm}) that $\nu_2(\phi_{i_m\ldots i_1}\Lambda) >0$, thus from the Borel Density Theorem (see \cite{Pe}), it follows that for $\nu_2$-a.e $x \in \phi_{i_m\ldots i_1}\Lambda$, and for all $0< r < r(x)$, $$\frac{\nu_2(B(x, r)\cap \phi_{i_m\ldots i_1}\Lambda)}{\nu_2(B(x, r))} > 1/2$$
 Hence, from the point of view of the measure $\nu_2$, the intersection $B(x, r) \cap \phi_{i_m\ldots i_1}\Lambda$ contains at least half of the $\nu_2$-measure of the ball $B(x, r)$. This hints to the fact that it is enough to consider only one ``good'' image set of type $\phi_{i_m\ldots i_1}\Lambda$. Then  since $\nu_2(\pi_2\tilde A(m, \vp)) \mathop{\to}\limits_m 1$, we can consider only $\nu_2(\phi_{i_m\ldots i_1}(\pi_2\tilde A(m, \vp)))$, which can be estimated using the Jacobian $J_{\Phi^m}(\hat\mu)$ and the genericity of points in $\tilde A(m, \vp)$ with respect to the functions $\log J_{\Phi^m}(\hat\mu)$ and $\log|\phi_{i_m\ldots i_1}'|$. Then we  repeat this argument whenever the iterate of a point belongs to the set of generic points $\tilde A(m, \vp)$. However not all iterates of a point belong to this set, but it will be shown by a delicate estimate that ''most'' of  them hit $\tilde A(m, \vp)$. For the iterates which do not belong to $\tilde A(m, \vp)$, we use a different type of estimate. Then we will repeat and combine these two types of estimates, by an interlacing procedure. \

We now proceed with the full detailed proof:
\newline
For any integer $m>1$, consider the Borel set $\tilde A(m, \vp)$ defined above. Then for any $\alpha>0$ arbitrarily small, there exists an integer $m(\alpha)>1$ such that for any $m > m(\alpha)$,  we have:
\begin{equation}\label{Amalpha}
\hat\mu(\tilde A(m, \vp)) > 1-\alpha
\end{equation}
Let us fix such an integer $m > m(\alpha)$. Then from Birkhoff Ergodic Theorem applied to $\Phi^m$ and $\chi_{\tilde A(m, \vp)}$, we have that for $\hat\mu$-a.e $(\omega', x')\in \Sigma_I^+\times \Lambda$, 
$$\frac 1n Card\{0 \le k \le n, \ \Phi^{km}(\omega', x') \in \tilde A(m, \vp)\} \mathop{\longrightarrow}\limits_{n \to \infty} \hat\mu(\tilde A(m, \vp))$$
Hence, there exists an integer $n(\alpha)$ and a Borel set $D(\alpha)\subset \Sigma_I^+\times \Lambda$, with $\hat\mu(D(\alpha)) > 1-\alpha$, such that for $(\omega', x') \in D(\alpha)$ and $n \ge n(\alpha)$, we have:
\begin{equation}\label{Dalpha}
\frac 1n Card\{0 \le k \le n, \ \Phi^{mk}(\omega', x') \in \tilde A(m, \vp)\} > 1-2\alpha
\end{equation}
In other words a large proportion of the iterates of points $(\omega', x')$ in $D(\alpha)$, belong to the set of generic points $\tilde A(m, \vp)$.
So in the $\Phi^m$-trajectory $(\omega', x'), \Phi^m(\omega', x'), \ldots, \Phi^{nm}(\omega', x')$, there are at least $(1-2\alpha)n$ iterates in $\tilde A(m, \vp)$.  \

 For arbitrary indices $i_1, \ldots  i_m \in I$, let us define now the Borel set in $\Lambda$, $$Y(i_1, \ldots, i_m):= \phi_{i_1 \ldots i_m}\pi_2(\tilde A(m, \vp))$$
Consider first the intersection of all these sets, namely $\mathop{\bigcap}\limits_{i_1, \ldots, i_m \in I} Y(i_1, \ldots, i_m)$. Then take the intersections of these sets except only one of them, so consider sets of type $\mathop{\bigcap}\limits_{(j_1, \ldots, j_m) \ne (i_1, \ldots, i_m)} Y(j_1, \ldots, j_m) \setminus Y(i_1, \ldots, i_m)$, for all $(i_1, \ldots, i_m) \in I^m$. Then we consider the intersections of all the sets $Y(j_1, \ldots, j_m)$ excepting two of them, namely the intersections of type $\mathop{\bigcap}\limits_{(j_1, \ldots, j_m) \notin \{(i_1, \ldots, i_m), \ (i_1', \ldots, i_m')\}} Y(j_1, \ldots, j_m) \setminus \big(Y(i_1, \ldots, i_m) \cup Y(i_1', \ldots, i_m')\big)$, for all the $m$-tuples $(i_1, \ldots, i_m),$ $(i_1', \ldots, i_m')\in I^m$. We continue this procedure until we exhaust all the possible intersections of type 
$$\mathop{\bigcap}\limits_{(j_1, \ldots, j_m) \notin \mathcal J} Y(j_1, \ldots, j_m) \setminus \mathop{\bigcup}\limits_{(i_1, \ldots, i_m) \in \mathcal J} Y(i_1, \ldots, i_m),$$ for some arbitrary given set $\mathcal J$ of $m$-tuples  from $I^m$.  Notice that in this way, by taking all the subsets $\mathcal J \subset I^m$, we obtain by the above procedure mutually disjoint Borel sets (some may be empty).  Denote these  mutually disjoint nonempty sets obtained above by $Z_1(\alpha; m, \vp), \ldots, Z_{M(m)}(\alpha; m, \vp)$. 

Now if for some $1\le i \le M(m)$, we know that $\nu_2(Z_i(\alpha; m, \vp)) >0$, then from the Borel Density Theorem (see \cite{Pe}),  there exists a Borel subset $G_i(\alpha; m, \vp) \subset Z_i(\alpha; m, \vp)$, with $\nu_2(G_i(\alpha; m, \vp)) \ge \nu_2(Z_i(\alpha; m, \vp))(1 - \alpha)$, and  there exists $r_i(\alpha; m, \vp)>0$, such that for any $x \in G_i(\alpha; m, \vp)$,
\begin{equation}\label{Gi}
\frac{\nu_2(B(x, r) \cap Z_i(\alpha; m, \vp))}{\nu_2(B(x, r))} > 1-\alpha,
\end{equation}
for any $0 <r <r_i(\alpha; m, \vp)$. \ 
Now define the Borel subset of $\Lambda$, 
$$G(\alpha; m, \vp):= \mathop{\bigcup}\limits_{i=1}^{M(m)} G_i(\alpha; m, \vp)$$
Now from the construction of the mutually disjoint sets $Z_i(\alpha; m, \vp)$, it follows that,
\begin{equation}\label{cupZ}
\mathop{\sum}\limits_{1 \le i \le M(m)} \nu_2(Z_i(\alpha; m, \vp)) = \nu_2(\mathop{\cup}\limits_{i_1, \ldots, i_m \in I} Y(i_1, \ldots, i_m))
\end{equation}
But from definition of $Y(i_1, \ldots, i_m)$ and the disjointness of different $m$-cylinders,  we have that,
$$
\begin{aligned}
\mathop{\cup}\limits_{i_1, \ldots, i_m \in I} Y(i_1, \ldots, i_m) & = \mathop{\cup}\limits_{i_1, \ldots, i_m \in I} \pi_2(\Phi^m([i_m\ldots i_1] \times \pi_2(\tilde A(m, \vp)))) =\\
&=  \pi_2(\mathop{\cup}\limits_{i_1, \ldots, i_m \in I} \Phi^m([i_m\ldots i_1]\times \pi_2\tilde A(m, \vp))) = \pi_2(\Phi^m(\Sigma_I^+ \times \pi_2\tilde A(m, \vp)))
\end{aligned}
$$
However since $\nu_2 = \pi_{2*}\hat\mu$, and using the $\Phi$-invariance of $\hat \mu$ and (\ref{Amalpha}), it follows that: $$
\begin{aligned}
\nu_2\big( \pi_2(\Phi^m(\Sigma_I^+ \times \pi_2\tilde A(m, \vp)))\big) &= \hat \mu(\Sigma_I^+ \times  \pi_2(\Phi^m(\Sigma_I^+ \times \pi_2\tilde A(m, \vp)))) \ge \hat\mu(\Phi^m(\Sigma_I^+ \times \pi_2\tilde A(m, \vp))) \ge \\ &\ge \hat\mu(\Sigma_I^+ \times \pi_2\tilde A(m, \vp))  \ge \hat \mu(\tilde A(m, \vp)) > 1-\alpha
\end{aligned}
$$
Thus from the last two displayed formulas, it follows that:
\begin{equation}\label{tildemare}
\nu_2(\mathop{\bigcup}\limits_{i_1, \ldots, i_m \in I} Y(i_1, \ldots, i_m)) > 1-\alpha
\end{equation}
Hence from (\ref{cupZ}) and (\ref{tildemare}), we obtain:
\begin{equation}\label{MmZ}
\nu_2\big(\mathop{\bigcup}\limits_{1 \le i \le M(m)} Z_i(\alpha; m, \vp) \big) > 1-\alpha
\end{equation}
But the Borel sets $Z_i(\alpha; m, \vp), 1 \le i \le M(m)$ are mutually disjoint, and from (\ref{Gi}) we know that $G_i(\alpha; m, \vp) \subset Z_i(\alpha; m, \vp)$ and for $1 \le i \le M(m)$, $$\nu_2(G_i(\alpha; m, \vp)) \ge  \nu_2(Z_i(\alpha; m, \vp)) (1-\alpha)$$  Hence from the definition of $G(\alpha; m, \vp) = \mathop{\bigcup}\limits_{1\le i \le M(m)} G_i(\alpha; m, \vp)$ and from (\ref{MmZ}),
\begin{equation}\label{nu2Ga}
 \nu_2(G(\alpha; m, \vp)) > (1-\alpha)^2> 1-2\alpha
 \end{equation}
 Denote now the following intersection set by,
$$X(\alpha; m, \vp):= G(\alpha; m, \vp) \cap \pi_2\tilde A(m, \vp)$$
Then from the above estimate for $\nu_2(G(\alpha; m, \vp))$ and using (\ref{Amalpha}), we obtain:
\begin{equation}\label{nu2X}
\nu_2(X(\alpha; m, \vp)) > 1-3\alpha
\end{equation}

Now by applying the same argument as in (\ref{Dalpha}) to the set $\tilde A(m, \vp) \cap \big(\Sigma_I^+\times X(\alpha; m, \vp)\big)$, we obtain that there exists a Borel set $\tilde D(\alpha; m, \vp) \subset \tilde A(m, \vp) \subset \Sigma_I^+\times \Lambda$, with 
\begin{equation}\label{Dam}
\hat\mu(\tilde D(\alpha; m, \vp)) > 1- 4\alpha,
\end{equation}
 and such that for any pair $(\ul i, x') \in \tilde D(\alpha; m, \vp)$, at least a number of  $(1-3\alpha)n$ of the points $\pi_2(\ul i, x')$, $\pi_2\Phi^m(\ul i, x'), \ldots, \pi_2\Phi^{nm}(\ul i, x')$ belong to  $X(\alpha; m, \vp)$. Moreover any point $\zeta \in X(\alpha; m, \vp)$ satisfies condition (\ref{Gi}) for all  $0 < r < r_i(\alpha; m, \vp)$ if $\zeta \in Z_i(\alpha; m, \vp)$, for some $1 \le i \le M(m)$. So denote $$r_m(\alpha, \vp) := \mathop{\min}\limits_{1\le i \le M(m)} r_i(\alpha; m, \vp)$$

Consider now $(\ul i, x') \in \tilde D_m(\alpha, \vp)$, and denote by $x = \phi_{i_{nm}\ldots i_1}(x') = \pi_2\Phi^{nm}(\ul i, x')$. So we have the following backward trajectory of $x$ with respect to $\Phi^m$ determined by the sequence $\ul i$ from above,  
\begin{equation}\label{seqx}
x, \phi_{i_{(n-1)m}\ldots i_1}(x'), \ldots, \phi_{i_m\ldots i_1}(x'), x',
\end{equation}
 and denote these points respectively by $x, x_{-m}, \ldots, x_{-(n-1)m}, x_{-nm} = x'$.  
To see the next argument,  assume for simplicity that the first preimage of $x$ in this trajectory, namely $x_{-m} = \phi_{i_{(n-1)m}\ldots i_1}(x')$ belongs to $X(\alpha; m, \vp)$. Then from (\ref{Gi}), Lemma \ref{prodstr}, and the genericity of $J_{\Phi^m}$ on $\tilde A(m, \vp)$,
$$
\begin{aligned}
\nu_2(B(x, r)) &\le \frac{1}{1-\alpha} \nu_2(B(x, r) \cap \phi_{i_{nm}\ldots i_{(n-1)m}}\pi_2\tilde A(m, \vp))  = \\
& = \frac{1}{1-\alpha} \hat\mu(\Sigma_I^+\times (B(x, r) \cap \phi_{i_{nm}\ldots i_{(n-1)m}}\pi_2\tilde A(m, \vp))) \le  \\
& \le \frac{1}{1-\alpha} \hat\mu(\Phi^m([i_{(n-1)m}\ldots i_{nm}]\times (\pi_2\tilde A(m, \vp)\cap  (\phi_{i_{nm}\ldots i_{(n-1)m}})^{-1}B(x, r) ))) = \\ 
&= \frac{1}{1-\alpha}\int_{[i_{(n-1)m}\ldots i_{nm}]\times (\pi_2\tilde A(m, \vp)\cap   (\phi_{i_{nm}\ldots i_{(n-1)m}})^{-1}B(x, r)  )} J_{\Phi^m}(\hat\mu) \ d\hat\mu \le \\
&\le C\frac{1}{1-\alpha} e^{m(F_\Phi(\hat\mu)+\vp)}\mu^+([i_{(n-1)m \ldots i_{nm}}])\cdot \nu_2( (\phi_{i_{nm}\ldots i_{(n-1)m}})^{-1}B(x, r) ) \le \\
&\le \frac{C}{1-\alpha} e^{m(F_\Phi(\hat\mu)+2\vp-h(\mu^+))}\nu_2( (\phi_{i_{nm}\ldots i_{(n-1)m}})^{-1}B(x, r)),
\end{aligned}
$$
where the last inequality follows since $\mu^+$ is the equilibrium state of $\psi$ and $P(\psi) = h_{\mu^+} + \int \psi d\mu^+$.
Thus we obtain from above the following estimate on the measure of $B(x, r)$,
\begin{equation}\label{nu2B}
\nu_2(B(x, r)) \le C \frac{1}{1-\alpha} \cdot e^{m(F_\Phi(\hat\mu)-h(\mu^+)+2\vp)}\cdot \nu_2( (\phi_{i_{nm}\ldots i_{(n-1)m}})^{-1}B(x, r) )
\end{equation}
This argument can be repeated until we reach in the above backward trajectory of $x$ (\ref{seqx}), a preimage  which is \textit{not} in $X(\alpha; m, \vp)$.  Denote then by $k_1\ge 1$ the first integer $k$ for which $x_{-mk} \notin X(\alpha; m, \vp)$, and assume that the above process is interrupted for $k_1'$ indices, namely $\{x_{-mk_1}, x_{-m(k_1+1)}, \ldots, x_{-m(k_1+k_1'-1)}\} \cap X(\alpha; m, \vp) = \emptyset$.  Then denote $y:= x_{-mk_1}$, and let us estimate $\nu_2( (\phi_{i_{nm}\ldots i_{(n-k_1)m}})^{-1}B(x, r) )$.
By definition of the projection measure $\nu_2$, 
$$\nu_2( (\phi_{i_{nm}\ldots i_{(n-k_1)m}})^{-1}B(x, r) ) = \hat\mu(\Sigma_I^+\times  (\phi_{i_{nm}\ldots i_{(n-k_1)m}})^{-1}B(x, r) )$$
 By definition of $k_1, k_1'$, we have that $x_{-m(k_1+k_1')} \in X(\alpha; m, \vp)$.
By repeating the estimate in (\ref{nu2B}), we obtain an upper estimate for $\nu_2(B(x, r))$, 
\begin{equation}\label{repeatnu2}
 \nu_2(B(x, r)) \le ( \frac{C}{1-\alpha})^{k_1} \cdot e^{mk_1(F_\Phi(\hat\mu)-h(\mu^+)+2\vp)}\cdot \nu_2( (\phi_{i_{nm}\ldots i_{(n-k_1)m}})^{-1}B(x, r) )
 \end{equation}
Now on the other hand from the $\Phi$-invariance of $\hat\mu$ and  the definition of $\nu_2$,
\begin{equation}\label{Byr1}
\begin{aligned}
\nu_2( (\phi_{i_{nm}\ldots i_{(n-k_1)m}})^{-1} B(x, r) ) &= \hat\mu(\Sigma_I^+\times  (\phi_{i_{nm}\ldots i_{(n-k_1)m}})^{-1}B(x, r) ) = \\
&= \hat\mu(\phi^{-mk_1'}(\Sigma_I^+\times  (\phi_{i_{nm}\ldots i_{(n-k_1)m}})^{-1}B(x, r) ))\\
&= \mathop{\sum}\limits_{j_p\in I, 1\le p \le mk_1'} \hat\mu([j_1\ldots j_{mk_1'}]\times \phi_{j_{mk_1'}\ldots j_1}^{-1}  (\phi_{i_{nm}\ldots i_{(n-k_1)m}})^{-1}B(x, r)  )
\end{aligned}
\end{equation}
Recall now that the set of non-generic points satisfies, $$\hat\mu((\Sigma_I^+\times \Lambda ) \setminus \tilde A(m, \vp)) < \alpha$$
We now compare the $\hat\mu$-measure of the set of generic points with respect to the $\hat\mu$-measure, with the $\hat\mu$-measure of the set of non-generic points. 
There are 2 cases: 

\ a)  \ If, $$\hat\mu\big(\tilde A(m, \vp) \bigcap \Phi^{-mk_1'}(\Sigma_I^+\times  (\phi_{i_{nm}\ldots i_{(n-k_1)m}})^{-1}B(x, r))\big) < \frac 12 \hat\mu(\Phi^{-mk_1'}(\Sigma_I^+\times  (\phi_{i_{nm}\ldots i_{(n-k_1)m}})^{-1}B(x, r) )),$$ then non-generic points have more mass than the generic points, hence,
$$
\nu_2( (\phi_{i_{nm}\ldots i_{(n-k_1)m}})^{-1}B(x, r) ) = \hat\mu(\Sigma_I^+ \times  (\phi_{i_{nm}\ldots i_{(n-k_1)m}})^{-1}B(x, r) ) < 2\alpha <<1
$$
By collecting all sets with the above property and taking $\alpha\to 0$, this case is then straightforward.

\ \  b) \  If, \ $$\hat\mu\big(\tilde A(m, \vp) \cap \Phi^{-mk_1'}(\Sigma_I^+\times  (\phi_{i_{nm}\ldots i_{(n-k_1)m}})^{-1}B(x, r) )\big) \ge \frac 12 \hat \mu(\Phi^{-mk_1'}(\Sigma_I^+\times  (\phi_{i_{nm}\ldots i_{(n-k_1)m}})^{-1}B(x, r) )),$$ then using also (\ref{Byr1}) we obtain:
\begin{equation}\label{genericmu2}
\begin{aligned}
\nu_2( (\phi_{i_{nm}\ldots i_{(n-k_1)m}})^{-1}B(x, r) ) &= \hat\mu(\Sigma_I^+\times  (\phi_{i_{nm}\ldots i_{(n-k_1)m}})^{-1}B(x, r) ) = \\ &= \hat\mu(\Phi^{-mk_1'}(\Sigma_I^+\times  (\phi_{i_{nm}\ldots i_{(n-k_1)m}})^{-1}B(x, r) )) \\
& \le 2\mathop{\sum}\limits_{\ul j \ \text{generic}}\hat\mu\big(\tilde A(m, \vp)\cap \big([j_1\ldots j_{mk_1'}] \times \phi_{j_{mk_1'}\ldots j_1}^{-1}  (\phi_{i_{nm}\ldots i_{(n-k_1)m}})^{-1}B(x, r)\big) \big)
\end{aligned}
\end{equation}
But for two generic histories $\ul j, \ul j'$, i.e. for $\ul j, \ul j' \in \pi_1(\tilde A(m, \vp))$, we know that for any 
 \ $(\omega, z) \in \tilde A(m, \vp) \bigcap \big( [j_1\ldots j_{mk_1'}]\times \phi_{j_{mk_1'}\ldots j_1}^{-1} (\phi_{i_{nm}\ldots i_{(n-k_1)m}})^{-1}B(x, r) $ $\bigcup \ [j_1'\ldots j_{mk_1'}']\times \phi_{j_{mk_1'}'\ldots j_1'}^{-1} (\phi_{i_{nm}\ldots i_{(n-k_1)m}})^{-1}B(x, r) \big)$,
$$J_{\Phi^{mk_1'}}(\hat\mu)(\omega, z) \  \in  \ \big(e^{mk_1'(F_\Phi(\hat\mu) -\vp)}, \ e^{mk_1'(F_\Phi(\hat\mu) +\vp)}\big).$$  Hence since $\Phi^{mk_1'}(\tilde A(m, \vp) \ \bigcap \ [j_1\ldots j_{mk_1'}]\times \phi_{j_{mk_1'}\ldots j_1}^{-1} (\phi_{i_{nm}\ldots i_{(n-k_1)m}})^{-1}B(x, r) \subset \phi_{i_{nm}\ldots i_{(n-k_1)m}})^{-1}B(x, r)$ and using the above estimate on the Jacobian of $\hat\mu$ with respect to $\Phi^{mk_1'}$, it follows that    there exists a constant factor $C>1$ such that the ratio of the $\hat\mu$-measures of the two preimage type sets   $$ \tilde A(m, \vp) \ \bigcap \ [j_1\ldots j_{mk_1'}]\times \phi_{j_{mk_1'}\ldots j_1}^{-1} (\phi_{i_{nm}\ldots i_{(n-k_1)m}})^{-1}B(x, r),  $$ and  $$ \tilde A(m, \vp) \ \bigcap \ [j_1'\ldots j'_{mk_1'}]\times \phi_{j'_{mk_1'}\ldots j'_1}^{-1} (\phi_{i_{nm}\ldots i_{(n-k_1)m}})^{-1}B(x, r),$$ corresponding to generic $\ul j, \ul j'$, \ belongs to the interval $\big(C^{-1} e^{-mk_1' \vp}, C e^{mk_1' \vp}\big)$. 
\newline
Notice also that there are at most $d^{mk_1'}$  sets of type $ [j_1\ldots j_{mk_1'}]\times \phi_{j_{mk_1'}\ldots j_1}^{-1} (\phi_{i_{nm}\ldots i_{(n-k_1)m}})^{-1}B(x, r)$, where $d:= |I|$. The maximality assumption for $k_1'$ implies that $\phi_{i_{m(n-k_1)}\ldots i_{m(n-k_1-k_1')}}^{-1}(y) \in X(\alpha; m, \vp)$, where $y:= x_{-mk_1}$. 
Thus using the above discussion and (\ref{genericmu2}), we obtain:
\begin{equation}\label{modphinu}
\nu_2((\phi_{i_{nm}\ldots i_{(n-k_1)m}})^{-1}B(x, r)) \le Cd^{mk_1'(1+\frac{2\vp}{\log d})}\cdot \nu_2(\phi_{i_{m(n-k_1-1)}\ldots i_{m(n-k_1-k_1')}}^{-1} (\phi_{i_{nm}\ldots i_{(n-k_1)m}})^{-1}B(x, r))
\end{equation}

Now, recall we concluded above that $\phi_{i_{m(n-k_1)}\ldots i_{m(n-k_1-k_1')}}^{-1}(y) \in X(\alpha; m, \vp)$, and then we can apply again the argument from (\ref{repeatnu2}) along the sequence $\ul i$ until we reach another preimage of $x$ which does not belong to $X(\alpha; m, \vp)$; after that we apply again the argument from (\ref{modphinu}), and so on, until reaching the $nm$-preimage of $x$, namely $x_{-nm}= \phi_{i_{nm}\ldots i_1}^{-1}(x)$. Recall from (\ref{seqx}) that  $x' := x_{-nm}$. Consider now $r_0>0$ to be a fixed radius, and $n= n(r)$ be chosen such that due to the conformality of $\phi_i, i \in I$,  one can assume that:
\begin{equation}\label{zeta0}
\phi^{-1}_{i_{nm}\ldots i_1}B(x, r) = B(x_{-nm}, r_0) = B(x', r_0)
\end{equation} 
On the backward $m$-trajectory (\ref{seqx}) of $x$ determined by $\ul i$ above, recall that we denoted by $k_1'$ the length of the first maximum ``gap'' consisting of preimages of $x$ which are not in $X(\alpha; m, \vp)$. Now let us denote in general the lengths of such maximal ``gaps'' in this trajectory (\ref{seqx}), consisting of consecutive preimages which do not belong to $X(\alpha; m, \vp)$, by  $k_1', k_2', \ldots$.  
\ \ More precisely, we have $x_{-jm} \in X(\alpha; m, \vp), 0\le j \le k_1-1$, followed by  $x_{-mk_1}, \ldots, x_{-m(k_1+k_1'-1)} \notin X_m$; then $x_{-m(k_1+k_1')}, \ldots, x_{-m(k_1+k_1'+k_2-1)} \in X_{m}(\alpha, \vp)$, followed by  $x_{-m(k_1+k_1'+k_2)}, \ldots, x_{-m(k_1+k_1'+k_2+k_2'-1)} \notin X(\alpha; m, \vp)$; then,  $x_{-m(k_1+k_1'+k_2+k_2')}, \ldots, x_{-m(k_1+k_1'+k_2+k_2'+k_3-1)}$ $\in X(\alpha; m, \vp)$, and so on.  
\newline
Denote thus by $k_1, \ldots, k_{p(n)}$ and $k_1', \ldots, k_{p'(n)}'$ all the integers obtained by the above procedure, corresponding to  the sequence (\ref{seqx}). Clearly we have $$k_1+k_1' + \ldots + k_{p(n)} + k_{p(n)}' = n$$ Also $p'(n)$ is equal either to $p(n)$ or to $p(n) -1$.
Then from the properties of $\tilde D_m(\alpha, \vp)$ in (\ref{Dam}), 
\begin{equation}\label{gapsbound}
k_1 + \ldots + k_{p(n)} \ge n(1-3\alpha), \ \text{and} \ \  k_1'+k_2'+\ldots + k_{p'(n)}'\le 3\alpha n
\end{equation}
We will apply (\ref{repeatnu2}) for the generic preimages $x, \ldots, x_{-m(k_1-1)}$ from the trajectory (\ref{seqx}), then we apply the estimate in (\ref{modphinu}) for the nongeneric preimages $x_{-mk_1}, \ldots, x_{-m(k_1+k_1'-1)}$, then again apply (\ref{repeatnu2}) for $x_{-m(k_1+k_1')}, \ldots, x_{-m(k_1+k_1'+k_2-1)}$, followed by (\ref{modphinu}) for $x_{-m(k_1+k_1'+k_2)}, \ldots, x_{-m(k_1+k_1'+k_2+k_2'-1)}$, and so on. 
Hence applying succesively the estimates  (\ref{repeatnu2}) and (\ref{modphinu}), and recalling (\ref{zeta0}) and the bound on the combined length of gaps $k_1'+k_2'+\ldots + k_{p'(n)}'\le 3 \alpha n$ from (\ref{gapsbound}), we obtain:
\begin{equation}\label{combest}
\nu_2(B(x, r)) \le (\frac{C}{1-\alpha})^{n}\cdot d^{3\alpha mn(1+\frac{\vp}{\log d})}\cdot e^{(F_\Phi(\hat \mu) - h(\mu^+) +2\vp) \cdot mn(1-3\alpha)} \nu_2(B(x', r_0))
\end{equation}
But  $\chi_s(\hat\mu)$ is the stable Lyapunov exponent of $\hat\mu$, i.e $\chi_s(\hat\mu) = \int_{\Sigma_I^+\times \Lambda} \log |\phi_{\omega_1}'(x)| \ d\hat\mu(\omega, x)$. From (\ref{zeta0}) it follows  that $r = r_n= r_0|\phi_{i_{nm}\ldots i_1}'(x')|$, and recall that $(\ul i, x') \in \tilde D_m(\alpha, \vp)$. Therefore, 
\begin{equation}\label{rlyap}
e^{nm(\chi_s(\hat\mu)-\vp)} \le r_n \le e^{nm(\chi_s(\hat\mu) +\vp)}
\end{equation}
Denote by $\tilde X(\alpha; m, \vp):= \pi_2 \tilde D_m(\alpha, \vp)$. Then since $\nu_2:= \pi_{2*}\hat\mu$ and $\hat\mu(\tilde D_m(\alpha, \vp)) > 1-4\alpha$, we obtain,
$$\nu_2(\tilde X(\alpha; m, \vp)) = \hat\mu(\Sigma_I^+\times \tilde X(\alpha; m, \vp)) \ge \hat \mu(\tilde D_m(\alpha, \vp)) \ge 1-4\alpha$$
From (\ref{combest}) and (\ref{rlyap}) it follows that by taking possibly $C' = 2C>1$, then for every $x \in \tilde X(\alpha; m, \vp)$ and for a sequence $r_n\to 0$, the following estimate holds: 
$$
\nu_2(B(x, r_n)) \le (\frac{C'}{1-\alpha})^{n}\cdot d^{3\alpha mn (1+\frac{2\vp}{\log d})} \cdot r_n^{(1-3\alpha)\cdot \frac{F_\Phi(\hat\mu)-h(\mu^+)+2\vp}{\chi_s(\hat\mu)}}
$$
Thus,
\begin{equation}\label{Thus14}
\log \nu_2(B(x, r_n)) \le (1-3\alpha) \log r_n\cdot \frac{F_\Phi(\hat\mu)-h(\mu^+)+2\vp}{\chi_s(\hat\mu)} + n \log \frac{C'}{1-\alpha} + 3\alpha mn \cdot \log d (1+\frac{2\vp}{\log d})
\end{equation}
But by using (\ref{rlyap}) and dividing in (\ref{Thus14}) by $\log r_n$, one obtains:
\begin{equation}\label{nu2r}
\frac{\log \nu_2(B(x, r_n))}{\log r_n} \ge (1-3\alpha) \frac{F_\Phi(\hat\mu)-h(\mu^+)+2\vp}{\chi_s(\hat\mu)} + \frac{\log\frac{C'}{1-\alpha}}{m(\chi_s(\hat\mu)+\vp)} + (1+\frac{2\vp}{\log d}) \frac{3 \alpha \log d}{\chi_s(\hat\mu) + \vp}
\end{equation}
Now let us take some arbitrary radius $\rho>0$, and assume that for some integer $n$, we have $r_{n+1}\le \rho \le r_n$, where $r_n$ is defined at (\ref{rlyap}). Then $\nu_2(B(x, \rho)) \le \nu_2(B(x, r_n))$, therefore $$\frac{\log\nu_2(B(x, \rho))}{\log \rho} \ge \frac{\log \nu_2(B(x, r_n))}{\log \rho}$$ But $\log \rho \ge \log r_{n+1}$, hence $\frac{1}{\log \rho} \le \frac{1}{\log r_{n+1}} < 0$, hence from above,
\begin{equation}\label{rhon}
\frac{\log\nu_2(B(x, \rho))}{\log \rho} \ge \frac{\log\nu_2(B(x, r_n))}{\log r_{n+1}} \ge \frac{\log \nu_2(B(x, r_n))}{c+\log r_n},
\end{equation}
since $r_{n+1} > cr_n$, for some constant $c$ independent of $n$.
Thus by letting $\rho \to 0$, and using (\ref{nu2r}) and (\ref{rhon}), we obtain the following lower estimate for the lower pointwise dimension of $\nu_2$:
$$\underline{\delta}(\nu_2)(x) \ge (1-3\alpha) \frac{F_\Phi(\hat\mu)-h(\mu^+)+2\vp}{\chi_s(\hat\mu)} + \frac{\log\frac{C'}{1-\alpha}}{m(\chi_s(\hat\mu)+\vp)} + 3\alpha(1+\frac{2\vp}{\log d}) \frac{\log d}{\chi_s(\hat\mu) + \vp}
$$
But  $\alpha \to 0$ when $m \to \infty$, and then $\nu_2(\tilde X(\alpha; m, \vp)) \to 1$. 
Thus from the last displayed estimate, it follows that for $\nu_2$-a.e $x \in \Lambda$, 
$$\underline{\delta}(\nu_2)(x) \ge  \frac{F_\Phi(\hat\mu)-h(\mu^+)+2\vp}{\chi_s(\hat\mu)} $$
Since $\vp$ is arbitrarily small, it follows from the above lower estimate and  (\ref{upperestd}),  that for $\nu_2$-a.e $x \in \Lambda$, 
$$\delta(\nu_2)(x) =  \frac{F_\Phi(\hat\mu)-h(\mu^+)}{\chi_s(\hat\mu)}$$
$\hfill\square$

\textit{Proof of Theorem \ref{thm2}.}

We proved in \cite{MU-JSP2016} that if $\mu^+$ is a Bernoulli measure $\nu_{\textbf p}$ on $\Sigma_I^+$, then the corresponding projections $\nu_{1, \textbf p}$ and $\nu_{2, \textbf p}$ are equal.
Hence they have the same pointwise dimensions. Also recall from \cite{FH} that $\nu_{1, \textbf p}$ is exact dimensional, thus also $\nu_{2, \textbf p}$ is exact dimensional. 
Thus, by using the formula (\ref{FHformula}) and Theorem \ref{thm1}, we obtain the expression for $h_\pi(\sigma, \nu_{\textbf p})$ in terms of the overlap number of $\hat\mu_{\textbf p}$, i.e, $$h_\pi(\sigma, \nu_{\textbf p}) = h_\sigma(\nu_{\textbf p}) - \exp(o(\mathcal S, \hat\mu_{\textbf p}))$$
In case $\nu_{\textbf p}$ is the measure of maximal entropy $\mu_0$ on $\Sigma_I^+$, we have $h(\mu_0) = \log |I|$, and we obtain $h_\pi(\sigma, \mu_0) = \log |I| - \exp(o(\mathcal S))$.

$\hfill\square$

In particular, we obtain a geometric formula for the pointwise dimension of the \textit{Bernoulli convolution} $\nu_\lambda$ in terms of its overlap number, for all $\lambda \in (\frac 12, 1)$. 

The Bernoulli convolution measure $\nu_\lambda$ is obtained as the projection of the measure of maximal entropy $\hat\mu_0$ onto the limit set $\Lambda_\lambda$ of the system $\mathcal S_\lambda$ given by the two contractions $\phi_1(x) = \lambda x - 1, \phi_2(x) = \lambda x +1$ on $\mathbb R$ (see \cite{PSS}). When $\lambda \in (\frac 12, 1)$, the system $\mathcal S_\lambda$ has overlaps, and its limit set $\Lambda_\lambda$ is the interval $I_\lambda = [-\frac{1}{1-\lambda}, \frac{1}{1-\lambda}]$. The measure $\nu_\lambda$ is the unique self-conformal measure satisfying $\nu_\lambda = \frac 12 \nu_\lambda \circ \phi_1^{-1} + \frac 12 \nu_\lambda \circ \phi_2^{-1}$. It is clear that in this case, if $\nu_{(\frac 12 \frac 12)}$ is the Bernoulli measure on $\Sigma_2^+$ given by the probability vector $(\frac 12, \frac 12)$, and if $\hat\mu_0$ is the measure of maximal entropy on $\Sigma_2^+\times \Lambda$, then  $$\nu_\lambda = \pi_*\nu_{(\frac 12, \frac 12)} = \pi_{2*}\hat\mu_0$$
Thus $\nu_\lambda$ is exact dimensional, and by applying Theorem \ref{thm2} we obtain:

\begin{cor}\label{bern}
For any $\lambda \in (\frac 12, 1)$,  the Hausdorff dimension of the Bernoulli convolution $\nu_\lambda$ satisfies:
$$\delta(\nu_\lambda) = \delta(\nu_\lambda)(x) = \frac{\log\frac{2}{o(\mathcal S_\lambda)}}{|\log \lambda|}, \ \text{for} \ \nu_\lambda-a.e \ x $$
\end{cor}

\

\textbf{Acknowledgements:}   The author thanks Professor Yakov Pesin for interesting discussions during a visit at Pennsylvania State University. During work on this article Eugen Mihailescu was supported by grant PN III-P4-ID-PCE-2016-0823 from UEFISCDI. He also acknowledges the support of Institut des Hautes \'Etudes Sci\'entifiques, Bures-sur-Yvette, France, for a stay when  part of this paper was done.

\
 
 Eugen Mihailescu
 
Institute of Mathematics of the Romanian Academy, 
Calea Grivitei 21,  
Bucharest, Romania.

Email: \  Eugen.Mihailescu\@@imar.ro

Web: \ www.imar.ro/$\sim$mihailes

\end{document}